\documentclass[12pt,reqno]{amsart}

\usepackage{amsmath,amssymb,accents,amsthm}

\usepackage[latin1]{inputenc}
\usepackage[all]{xypic}

\newtheorem{theorem}{Theorem}[section]
\newtheorem{proposition}[theorem]{Proposition}
\newtheorem{corollary}[theorem]{Corollary} 
\newtheorem{lemma}[theorem]{Lemma} 
\theoremstyle{definition} 
\newtheorem{definition}[theorem]{Definition} 

\numberwithin{equation}{section} 

\newcommand\dl{\partial} 
\renewcommand\l{\ell}
\renewcommand\#{\sharp}
\renewcommand\S{\sum\limits}

\newcommand\la{\alpha}
\newcommand\lb{\beta}
\newcommand\lG{\Gamma}
\newcommand\Le{\epsilon}
\newcommand\lh{\eta}
\newcommand\li{\iota}
\newcommand\lk{\kappa}
\newcommand\lm{\mu}
\newcommand\lp{\pi}
\newcommand\lt{\vartheta}
\newcommand\lT{\Theta}
\newcommand\lo{\omega}
\newcommand\lO{\Omega}
\newcommand\lX{\Xi}
\newcommand\lz{\zeta}
\newcommand\ui{\cap}
\newcommand\xx{\times}
\newcommand\xt{\otimes}
\newcommand\oo{\infty}
\newcommand\oc{\circ}
\newcommand\CL{\mathcal C} 
\newcommand\LL{\mathcal L} 
\newcommand\Cl{{\mathbb C}} 
\newcommand\Gl{{\mathbf G}} 
\newcommand\Rl{{\mathbb R}}
\newcommand\Sl{{\mathbf S}} 
\newcommand\Tl{{\mathbf T}} 
\newcommand\Zl{{\mathbf Z}} 
\newcommand\al{{\boldsymbol a}} 
\newcommand\bl{{\boldsymbol b}} 
\newcommand\cl{{\boldsymbol c}} 
\newcommand\ql{{\boldsymbol q}} 
\newcommand\cc[9]{\begin{pmatrix}{#1}&{#2}&{#3}\\ {#4}&{#5}&{#6}\\ {#7}&{#8}&{#9}\end{pmatrix}}
\renewcommand\d[1]{{\dot{#1}}}
\newcommand\h[1]{{\widehat{#1}}} 
\newcommand\y[1]{{\undertilde{#1}}} 
\renewcommand\t[1]{{\widetilde{#1}}} 
\newcommand\w[1]{{\widetilde{#1}}} 
\renewcommand\u[1]{{\underline{#1}}}
\newcommand\f[2]{{\frac{#1}{#2}}}

\begin{document}

\title[Invariant connections on hermitian symmetric spaces]{A normal variety of invariant 
connections on hermitian symmetric spaces}

\author[I. Biswas]{Indranil Biswas}

\address{School of Mathematics, Tata Institute of Fundamental Research, Homi Bhabha Road, Bombay 
400005, India}

\email{indranil@math.tifr.res.in}

\author[H. Upmeier]{Harald Upmeier}

\address{Fachbereich Mathematik und Informatik, Philipps-Universit\"at Marburg, 
Lahnberge, Hans-Meerwein-Strasse, D-35032 Marburg, Germany}

\email{upmeier@mathematik.uni-marburg.de}



\date{}

\begin{abstract}
We introduce a class of $G$-invariant connections on a homogeneous principal bundle $Q$ over a hermitian symmetric 
space $M\,=\,G/K$. The parameter
space carries the structure of normal variety and has a canonical anti-holomorphic involution. The 
fixed points of the anti-holomorphic involution
are precisely the integrable invariant complex structures on $Q.$ This normal variety is closely related to 
quiver varieties and, more generally, to varieties of commuting matrix tuples
modulo simultaneous conjugation.
\end{abstract}

\maketitle

\tableofcontents

\section{Introduction}

In \cite{BU} we undertook a systematic study of $G$-invariant connections on principal 
$H$-bundles, for a complex algebraic group $H,$ over $G$-homogeneous spaces $G/K,$ in 
particular over hermitian symmetric spaces. The main result there gave an intrinsic ``curvature''
characterization of invariant connections which produce (invariant) integrable complex 
structures. In addition it was shown that this characterization holds in the both compact and 
non-compact hermitian symmetric cases.

In this paper, we introduce the concept of ``pure'' invariant connection (related to the 
curvature condition above) and regard the space of all such pure invariant connections as an 
interesting moduli space in the sense of algebraic geometry. We prove that this moduli space is 
an algebraic variety endowed with a canonical anti-holomorphic involution, and investigate the 
question whether this variety is normal. In dimension 1 (the hermitian
symmetric space is the unit disk) and, more generally, for 
rank 1 (the hermitian symmetric space is a unit ball), this problem is solved very explicitly, by realizing the moduli space as a 
quiver variety. For symmetric spaces of higher rank, we establish a close connection to 
varieties of commuting matrix tuples (modulo joint conjugation).

\section{Principal bundles}\label{sec2}

We refer to \cite{KN}, \cite{H}, \cite{B} for Lie groups and principal bundles.

As this paper is a continuation of our paper \cite{BU} we first recall some concepts and notation from
\cite{BU}. For a Lie group $H$ we consider $\CL^\oo$ principal $H$-bundles $Q$ over a given
manifold $M,$
with projection map $\lp\,:\,Q\,\longrightarrow\, M.$ For
another principal $H'$-bundle $\lp'\,:\,Q'\,\longrightarrow\, M'$, a
bundle morphism $F\, :\,Q\,\longrightarrow\, Q'$ is given by a commuting diagram
$$\xymatrix{Q\ar[r]^F\ar[d]_\lp&Q'\ar[d]^{\lp'}\\ M\ar[r]_g&M'}$$
of smooth maps, such that there exists a Lie group homomorphism $f\,:\,H\,\longrightarrow\, H'$ satisfying
the condition $$F(qh)\,=\,F(q)f(h)$$
for all $q\,\in\, Q,\ h\,\in\, H.$ We say that $F$ is a bundle map over $g,$ associated to
the homomorphism $f.$ If $F$ is a diffeomorphism, $M\,=\,M'$ and 
$g\,=\,id_M,$ we obtain a (bundle) isomorphism. 

Two special cases of bundle maps arise as follows.

\begin{definition} Let $H$ be a complex Lie group, and let $H_\Rl\,\subset\, H$ be a maximal compact
subgroup. A {\it hermitian structure} on a principal
$H$-bundle $Q$ is a subbundle $Q_\Rl\,\subset\, Q$ with structure group $H_\Rl.$ Thus the inclusion map 
$\li\,:\,Q_\Rl\,\hookrightarrow\, Q$ is a bundle morphism 
\begin{equation}\label{14}
\xymatrix{Q_\Rl\ar[rr]^{\li}\ar[dr]_{\lp}&&Q\ar[dl]^{\lp}\\ &M&}
\end{equation}
over $id_M,$ associated to the inclusion $H_\Rl\,\hookrightarrow\, H.$ A morphism $(Q,\,Q_\Rl)
\,\longrightarrow\,(Q',\,Q_\Rl')$ between hermitian structures over $M$ is a morphism $F\,:\,Q\,\longrightarrow\, Q'$
such that $F(Q_\Rl)\,\subset\, Q_\Rl'.$ Putting $F_\Rl\,:=\,F|_{Q_\Rl},$ we have a commuting diagram
$$\xymatrix{Q\ar[ddr]_\lp\ar[rrr]^F&&&Q'\ar[ddl]^{\lp'}\\ &Q_\Rl\ar[d]^\lp\ar[ul]_\li\ar[r]^{F_\Rl}&Q_\Rl'\ar[ur]^{\li'}\ar[d]_{\lp'}&\\ &M\ar[r]_g&M'&}$$
\end{definition}

Let $H$ be a connected reductive affine algebraic group defined over $\mathbb C$, and let $H_\Rl\,\subset
\, H$ be a maximal compact subgroup. Then there exists an anti-automorphism
$h\,\longmapsto\, h^*$ of order 2 on $H$ such that
$$H_\Rl\,=\,\{h\,\in\, H\,\mid\ h^*\,=\,h^{-1}\}\, .$$
So $H_\Rl$
consists of the ``unitary'' elements in $H.$ In other words, the corresponding
antiholomorphic involution $\#$ (automorphism of order 2)
$$h\,\longmapsto\, h^\#\,:=\,{h^*}^{-1}\,=\,{h^{-1}}^*$$
has the fixed point subgroup $H_\Rl.$ Its differential $\u{\#}\,=\,d_e\#$ is an anti-linear involution on the Lie algebra 
${\mathfrak h},$ preserving the Lie bracket operation.

\begin{definition} Let $H$ be a complex reductive group, and let $H_\Rl\,\subset\, H$ be a maximal compact subgroup. A
principal $H$-bundle $Q$ is called 
{\it involutive} if there exists a morphism of order two
\begin{equation}\label{1}\xymatrix{Q\ar[rr]^{\lk}\ar[dr]_{\lp}&&Q\ar[dl]^{\lp}\\ &M&}
\end{equation}
over ${id}_M,$ associated with the involution $h\,\longmapsto\, h^\#$ of $H.$ A morphism between involutive principal bundles $(Q,\lk)$ and 
$(Q',\lk')$ is described by a commuting diagram
$$\xymatrix{Q\ar[rrr]^F\ar[ddr]_\lp\ar[dr]^\lk&&&Q'\ar[ddl]^{\lp'}\ar[dl]_{\lk'}\\ 
&Q\ar[r]^F\ar[d]^\lp&Q'\ar[d]_{\lp'}&\\ &M\ar[r]_g&M'&}$$
\end{definition}

It turns out that these two concepts are equivalent, as shown by the following proposition.

\begin{proposition} For an involutive principal $H$-bundle $(Q,\,\lk),$ the fixed point set
\begin{equation}\label{2}
Q_\Rl\,=\,\{p\,\in\, Q\,\mid\ \lk(p)\,=\,p\}\end{equation}
is a subbundle which defines a hermitian structure on $Q.$ Conversely, every hermitian structure $(Q,\,Q_\Rl)$ induces
a bundle involution $\lk$ on $Q$ with fixed point subbundle $Q_\Rl.$
\end{proposition}

\begin{proof}
Consider the free action $Q\xx H\,\longrightarrow\, Q$ with quotient $Q/H\,=\,M.$ For
$p\,\in\, Q_\Rl,\, \l\,\in\, H_\Rl$ we have
$$\lk(p\l)\,=\,\lk(p)\l^\#\,=\,p\l$$
since $\lk$ is a bundle map and $Q_\Rl$ and $H_\Rl$ are fixed pointwise under the involutions
$\lk$ and $\#$ respectively. Thus we have a (free) action 
$Q_\Rl\xx H_\Rl\,\longrightarrow\, Q_\Rl.$ Clearly, $p\sim p'$ modulo $H_\Rl$ implies $p\,\sim\, p'$ modulo $H,$ so
there is a map $Q_\Rl/H_\Rl\,\longrightarrow\, Q/H.$ This map is injective, since $p'\,=\,ph$ for some $h
\,\in\, H$ implies that $$ph\,=\,p'\,=\,\lk(p')\,=\,\lk(ph)\,=\,\lk(p)h^\#\,=\, ph^\#\, .$$ Therefore,
we have $h\,=\,h^\#\in H_\Rl.$

In order to show that the above map $Q_\Rl/H_\Rl\,\longrightarrow\, Q/H$ is
surjective, let $q\,\in \,Q$ be arbitrary. Since $\lk$ preserves
base points, there exists a unique $k\,\in\, H$ such that 
$\lk q\,=\,qk.$ It follows that 
$$q\,=\,\lk(\lk q)\,=\,\lk(qk)\,=\,(\lk q)k^\#\,=\,qkk^\#\, .$$
Therefore, $kk^\#\,=\,1$ and $k\,=\,(k^\#)^{-1}\,\in\, H$ is ``unitary''. Hence there exists $h\,
\in\, H$ such that $k\,=\,h(h^\#)^{-1}.$ For this element $h$ we have
$\lk(qh)\,=\,(\lk q)h^\#\,=\,qkh^\#\,=\,qh.$ Hence $p\,:=\,qh\,\in\, Q_\Rl$ and $p\,\sim\, q$ mod $H.$ This
shows that $Q_\Rl/H_\Rl\,=\,Q/H.$ Thus every involutive principal 
$H$-bundle $Q$ gives rise to a hermitian structure.

Conversely, let $(Q,\,Q_\Rl)$ be a hermitian structure and let $q\,\in\, Q$ be arbitrary. Since $Q/H
\,=\,Q_\Rl/H_\Rl$ there exists $h\,\in\, H$ such that
$p\,:=\,qh\,\in\, Q_\Rl.$ Define a map $\lk\,:\,Q\,\longrightarrow\, Q$ by
\begin{equation}\label{3}
\lk q\,:=\,p(h^\#)^{-1}\,=\,qh(h^\#)^{-1}\, .
\end{equation}
In order to show that \eqref{3} is independent of the choice of $h,$ let $qh_1\,\in\, Q_\Rl$ for some $h_1\,\in\, H.$
Since $qh$ and $qh_1$ have the same base point in $M,$ there exists $\l\,\in\, H_\Rl$ such that $qh_1
\,=\,qh\l.$ Therefore, $h_1\,=\,h\l$, and hence 
$$h_1(h_1^\#)^{-1}\,=\,h\l((h\l)^\#)^{-1}\,=\,h\l(h^\#\l)^{-1}\,=\,h(h^\#)^{-1}\, .$$
It is now straightforward to show that $q\,\longmapsto\,\lk q$ is a bundle isomorphism of $Q$ of order $2$
whose fixed point locus is the subbundle $Q_\Rl.$
\end{proof}

\section{Equivariant principal $H$-bundles}

Let $G$ be a real Lie group acting on $M$ via transformations $L_g^M,\ g\in G.$ 

\begin{definition}\label{def-2} A $\CL^\oo$ principal $H$-bundle over $M$ is called {\it equivariant} (under $G$) if
$Q$ carries a smooth (left) action 
$g\,\longmapsto\, L_g^Q$ of $G$ which commutes with the right action of $H$ on $Q,$ such that for each $g\,\in\, G$
the diagram
\begin{equation}\label{37}\xymatrix{Q\ar[r]^{L_g^Q}\ar[d]_\lp&Q\ar[d]^\lp\\ M\ar[r]_{L_g^M}&M}\end{equation}
commutes. Thus $L_g^Q$ is a bundle morphism over $L_g^M,$ associated with $id_H.$ A bundle morphism 
$F\,:\,Q\,\longrightarrow\, Q'$ between equivariant $H$-bundles is called equivariant if it
intertwines the actions of $G\xx H.$
\end{definition}

\begin{definition} For a complex Lie group $H$ an involutive $H$-bundle $(Q,\,\lk)$ is called
{\it equivariant}, if $Q$ is equivariant and each $L_g^Q$ is an isomorphism of the involutive bundle $(Q,\,\lk),$ i.e.,
$$L_g^Q\oc\lk\,=\,\lk\oc L_g^Q$$
for every $g\,\in\, G.$ An equivalent condition is that the involution $\lk$ is an equivariant morphism.

Similarly, a hermitian structure 
$(Q,\,Q_\Rl)$ is called {\it equivariant}, if $Q$ is equivariant and each $L_g^Q$ is an isomorphism of
the hermitian structure $(Q,\,Q_\Rl),$ i.e., satisfies the condition
$L_g^QQ_\Rl\,=\,Q_\Rl.$ Then $Q_\Rl$ becomes an equivariant bundle under the restricted action 
$L_g^{Q_\Rl}\,:=\,L_g^Q|_{Q_\Rl},$ for which the inclusion map $\li\,:\,Q_\Rl\,\hookrightarrow\, Q$ is
an equivariant bundle morphism.
\end{definition}

Equivariant morphisms between equivariant involutive $H$-bundles or equivariant hermitian structures are defined in a natural way.

We shall apply these concepts to the homogeneous case, in particular when $M\,=\,G/K$ is an irreducible
symmetric space. For any $g\,\in\, G$, let
\begin{equation}\label{4}
L_g^M\,:\,G/K\, \longrightarrow\, G/K
\end{equation}
be the left translation defined by $L_g^M(g'K)\,:=\,gg'K$. In the homogeneous case an equivariant principal 
$H$-bundle $Q$ (and, similarly, an equivariant hermitian structure $(Q,\,Q_\Rl)$ or an involutive $H$-bundle 
$(Q,\,\lk)$) will also be called \textit{homogeneous}. By \cite[Theorem 1.2]{BU}, all equivariant homogeneous
principal $H$-bundles, 
up to equivariant isomorphism, are classified by the set $\mbox{Hom}(K,\, H)/H,$ by associating to a Lie group 
homomorphism $f\,:\,K\,\longrightarrow\, H,$ up to $H$-conjugacy, the homogeneous principal $H$-bundle 
\begin{equation}\label{5}
Q\,=\,G\xx_{K,f}H\,=\,\{[g:\,h]=[gk:\,f(k)^{-1}h]\,\mid\ g\,\in\, G,\ h\,\in\, H,\ k\,\in 
\,K\},\end{equation}
which is equivariant under the (well-defined) action 
\begin{equation}\label{8}
L_g^Q[g':\,h]\,:=\,[gg':\,h]\, .\end{equation}
More precisely, for any choice of base 
point $\ql\,\in \,Q_o$ in the fiber over $o\,=\,eK\in G/K,$ the homomorphism $f\,:\,K \,\longrightarrow\, H$ is uniquely 
determined by the condition
\begin{equation}\label{6}k\ql\,=\,\ql f(k)\end{equation}
for all $k\,\in\, K$; another choice $\ql'\,=\,\ql h^{-1}\,\in \,Q_o$ induces the conjugate homomorphism
$I_h^H\oc f,$ where 
$$I_h^Hh'\,=\,hh'h^{-1}$$ 
denotes the inner automorphism.

Similarly, all equivariant hermitian structures $(Q_\Rl,\,Q),$ up to equivariant isomorphism, are
parametrized by the discrete set 
$\mbox{Hom}(K,\,H_\Rl)/H_\Rl,$ mapping the $H_\Rl$-conjugacy class of a homomorphism $f\,:\,K
\,\longrightarrow\, H_\Rl$ to the equivariant $H_\Rl$-bundle
$$Q_\Rl\,=\,G\xx_{K,f}H_\Rl\,\subset\, Q\,=\,G\xx_{K,f}H$$
regarded as a subbundle of \eqref{5}. Since for any homomorphism $f\,:\,K\,\longrightarrow\, H,$ the image
is contained in a maximal compact subgroup 
$H_\Rl,$ and for reductive Lie groups all maximal compact subgroups $H_\Rl\,\subset\, H$ are conjugate, the two
classifying sets can be identified, i.e., any homogeneous $H$-bundle $Q$ admits a hermitian structure
$Q_\Rl\,\subset\, Q,$ which is unique after specifying the maximal compact subgroup $H_\Rl.$

\begin{lemma}
Let $H$ be a complex reductive group, and let $H_\Rl\,\subset\, H$ be a maximal compact subgroup. Then the
associated equivariant bundle $Q\,=\,G\xx_{K,f}H$ is involutive with respect to $H_\Rl,$ under the involution 
\begin{equation}\label{7}
\lk[g:\,h]\,:=\,[g:\,h^\#]
\end{equation}
for all $g\,\in\, G,\ h\,\in\, H.$ The associated hermitian structure has the fixed point subbundle
$$Q_\Rl\,=\,G\xx_{K,f}H_\Rl.$$
\end{lemma}

\begin{proof} For all $k\,\in\, K$ we have $f(k)\,\in\, H_\Rl$, and hence $f(k)^\#\,=\,f(k).$ It follows that
$$\lk[gk:\,f(k)^{-1}h]\,=\,[gk:\,(f(k)^{-1}h)^\#]\,=\,[gk:\,f(k)^{-1}h^\#]\,=\,[g:\,h^\#].$$
Therefore the map in \eqref{7} is well-defined. It is clear that it commutes with the $G$-action on $Q$
in \eqref{8}. 
\end{proof}

While the principal bundle in \eqref{5} depends only on the homomorphism $f\,:\,K\,\longrightarrow\, H,$ the base 
point $\ql$ satisfying \eqref{6} is not unique. Define a subgroup
\begin{equation}\label{9}
H^f\,:=\,\{\lh\in H\,\mid\,\ f(k)\lh\,=\,\lh 
f(k)\quad\forall k\in K\}\,=\,\{\lh\in H\,\mid\ I^H_\lh\oc f\,=\,f\}.
\end{equation}
Then for any $\lh\,\in\, H^f$, the point 
$\ql'\,:=\,\ql\lh$ also satisfies \eqref{6}, because $$k\ql'\,=\,k\ql\lh
\,=\,\ql f(k)\lh\,=\,\ql\lh f(k)\,=\,\ql'f(k)$$ for all $k\,\in\, K.$ For 
$f\,:\,K\,\longrightarrow\, H_\Rl$ we put
$$H_\Rl^f\,:=\,H_\Rl\ui H^f.$$

For $x\,\in \,M$ we define the evaluation map $R_x^M\,:\,G\,\longrightarrow\, M$ by
$$R_x^M(g)\,:=\,L_g^M(x)\,=\,gx\, .$$
Via the map $R_o^M\,:\,G\,\longrightarrow\, M$ the group $G$ can be regarded as a principal
$K$-bundle over $M\,=\,G/K,$ which is equivariant under the left translation action
$$L_g^Gg'\,:=\,gg'$$
for $g,\,g'\,\in\, G.$ This follows from the commuting diagram
$$\xymatrix{G\ar[r]^{L_g^G}\ar[d]_{R_o^M}&G\ar[d]^{R_o^M}\\ M\ar[r]_{L_g^M}&M}.$$

Sometimes we write $\mathbf G$ to emphasize the principal $K$-bundle structure. Now consider the homogeneous $K$-bundle
$$G\xx_K K\,:=\,\{[g:\,k]\,=\,[gk_1:\,k_1^{-1}k]\,\mid\ g\,\in\, G,\ k,\,k_1\,\in\, K\}$$
which is equivariant under the left action
$$g\cdot [g',\,k]\,=\,[gg',\,k]$$
for all $g,\,g'\,\in\, G$ and $k\,\in\, K.$ Then the map
$$G\,\ni\, g\,\longmapsto\,[g:\,e_H]\,=\,[gk:\,k^{-1}]\,\in\, G\xx_K K$$
defines an equivariant isomorphism of principal $K$-bundles
$$\xymatrix{G\ar[rr]\ar[dr]_{R_o^M}&&G\xx_K K\ar[dl]^\lp\\ &G/K&}$$

\begin{lemma} For a homomorphism $f\,:\,K\,\longrightarrow\, H$, and $\lh\,\in\, H^f$, define
$$f_\lh(g)\,:=\,[g:\,\lh]\,=\,g[e:\,\lh]$$ 
for all $g\,\in\, G.$ Then $f_\lh$ is an equivariant bundle map 
$$\xymatrix{G\ar[dr]_{R^M_o}\ar[rr]^{f_\lh}&&G\xx_{K,f}H\ar[dl]^\lp\\ &M&}$$ 
over the identity map of $M\,=\,G/K,$ associated with the homomorphism $f.$
\end{lemma}

\begin{proof} This follows from the identity
$$f_\lh(gk)\,=\,[gk:\,\lh]\,=\,[g:\,f(k)\lh]\,=\,[g:\,\lh f(k)]\,=\,[g:\,\lh]f(k)\,=\,f_\lh(g)f(k)$$
for all $g\,\in\, G,\ k\,\in\, K.$
\end{proof}

In case $f\,:\,K\,\longrightarrow\, H_\Rl$, and $\lh\,\in\, H_\Rl^f$, we obtain a similar $G$-equivariant bundle morphism 
$$\xymatrix{G\ar[dr]_{R^M_o}\ar[rr]^{f_\lh}&&G\xx_{K,f}H_\Rl\ar[dl]^\lp\\ &M&}$$
Thus the various choices of base point in $Q$ correspond to suitable bundle embeddings of $G$ into $Q.$

\begin{proposition} For every $\lh\,\in\, H^f$ the map
$$L_\lh^Q[g\,:h]\,:=\,[g:\,\lh h]$$
defines a bundle automorphism $L_\lh^Q$ of $Q\,=\,G\xx_{K,f}H$ satisfying $L_h^Q\ql\,=\,\ql'.$ 
\end{proposition}

\begin{proof}
For $k\,\in\, K$ we have $f(k)^{-1}\lh h\,=\,\lh f(k)^{-1}h.$ This shows that $L_h^Q$ is a well-defined bundle
map. Moreover,
$$L_h^Q\ql\,=\,L_h^Q[e:\,e_H]\,=\,[e:\,h]\,=\,[e:\,e_H]h\,=\,\ql h\,=\,\ql'\, ,$$
and the proposition is proved.
\end{proof}

There is a commuting diagram
$$\xymatrix{&Q\ar[dd]^{L_\lh^Q}\\ G\ar[ur]^{f_{\lh'}}\ar[dr]_{f_{\lh\lh'}}&\\ &Q}$$

\section{Connections}

For a given principal $H$-bundle $Q$ over $M,$ we consider {\bf connections} $\lT$ on $Q$.
The space of all connections on $Q$ will be denoted by $\CL(Q)$.
For a connection $\lT\, \in\, \CL(Q)$, let $T_q^\lT Q\,\subset\, T_qQ$ be the horizontal
tangent subspace at $q\,\in \,Q$.

We identify a connection on $Q$ with its {\bf connection 1-form} the total space of $Q$,
which is a pseudo-tensorial ${\mathfrak h}$-valued 1-form 
$$T_qQ\,\ni\, X\,\longmapsto\,\lT_qX\,\in\,{\mathfrak h},\ \ q\,\in\, Q$$
on $Q$, which is $H$-equivariant and coincides with the Maurer-Cartan form on the fibers
of $Q$ \cite[Section II.1]{KN}, \cite{A}, \cite{K}. Let 
$\lO^j_H(Q,\,{\mathfrak h})$ denote the vector space of all tensorial ${\mathfrak h}$-valued $j$-forms on $Q$ that
are $H$-equivariant. Given
a fixed connection $\lT^0\,\in\, \CL(Q)$ on $Q,$ all other connections on $Q$ are of the form
\begin{equation}\label{10}
\lT^\lo_q\,:=\,\lT^0_q+\lo_q,\quad q\,\in\, Q\, ,
\end{equation}
where $\lo\,\in\,\lO^1_H(Q,\,{\mathfrak h}).$ In short, we have
$\CL(Q)\,=\,\lT^0+\lO^1_H(Q,\,{\mathfrak h}).$ This notation depends on the choice of $\lT^0.$ Let 
$\u\lT\,\in\,\lO^2_H(Q,\,{\mathfrak h})$ denote the {\it curvature} of $\lT.$ By \cite[Section II.5]{KN}, for each
$\lO\in\lO^j_H(Q,\,{\mathfrak h})$ there exists a unique $j$-form $\w\lO$ on $M$ of type $Ad_H,$ i.e., with values in the homogeneous vector bundle
$$Q\xx_H{\mathfrak h}\,=\,\{[q:\,\lb]\,=\,[qh:\,Ad_{h^{-1}}^H\lb]\,\mid\ q\,\in \,Q,\ \lb\,\in\,{\mathfrak h},\ h\,\in\, H\}\, ,$$
such that
$$\w\lO_{\lp(q)}((d_q\lp)X_1,\,\cdots ,\,(d_q\lp)X_j)\,=\,[q:\,\lO_q(X_1,\,\cdots ,\,X_j)]$$
for all $q\,\in\, Q$ and $X_i\,\in \, T_qQ$, where $1\,\le\, i\,\le\, j.$ We call $\w\lO$ the
{\bf associated bundle-valued} $j$-form.

We shall often use the concept of {\bf induced connection}. Let
$$\xymatrix{Q\ar[r]^F\ar[d]_\lp&Q'\ar[d]^{\lp'}\\ M\ar[r]_g&M'}$$
be a principal bundle map over a diffeomorphism $g\,:\,M\,\longrightarrow\, M',$ associated to a Lie group homomorphism $f\,:\,H\,\longrightarrow\, H'.$ Then any connection 
$\lT$ on $Q$ induces a connection, denoted by $F_*\lT,$ on $Q',$ such that the bundle map $F$ preserves
the respective horizontal subspaces. By \cite[Proposition II.6.1]{KN} the connection and curvature forms
of $\lT$ and $F_*\lT$ are related by
\begin{equation}\label{11}
F^*(F_*\lT)\,=\,\u f\oc\lT
\end{equation}
\begin{equation}\label{12}
F^*(\u{F_*\lT})\,=\, \u f\oc\u\lT\, ,
\end{equation}
where $F^*$ denotes the pull-back of differential forms, and $\u f\,=\,d_ef\,:\,{\mathfrak h}\,\longrightarrow\,{\mathfrak h}'$
is the differential of $f$ at the unit element $e\,\in\, H.$ 

For a hermitian structure $(Q,\,Q_\Rl),$ we also consider ${\mathfrak h}_\Rl$-valued connection 1-forms $\lX$ on $P.$
Given a fixed connection 
$\lX^0$ on $Q_\Rl,$ all other connections on $Q_\Rl$ have the form
\begin{equation}\label{13}
\lX^\lo_p\,:=\,\lX^0_p+\lo_p,\quad p\,\in\, Q_\Rl\, ,
\end{equation}
where $\lo\,\in\,\lO^1_{H_\Rl}(Q_\Rl,\,{\mathfrak h}_\Rl).$ Thus we have
$\CL(Q_\Rl)\,=\,\lX^0+\lO^1_{H_\Rl}(Q_\Rl,\,{\mathfrak h}_\Rl).$
Under the inclusion map in \eqref{14} any connection $\lX$ on $Q_\Rl$ induces a unique connection
$\li_*\lX$ on $Q$ preserving horizontal subspaces.
Thus we have a natural map $\li_*\,:\,\CL(Q_\Rl)\,\longrightarrow\,\CL(Q)$ whose image
$$\CL_\Rl(Q)\,:=\,\li_*\big(\CL(Q_\Rl)\big)$$ 
consists of what are known as {\bf hermitian} connections on $Q.$ 

\begin{proposition} Let $\lT^0\,:=\,\li_*\lX^0.$ Then
$$\li_*\lX^\lo\,=\,\lT^{\u\li\oc\lo}$$
for any $\lo\,\in\,\lO^1_{H_\Rl}(Q_\Rl,\,{\mathfrak h}_\Rl),$ where $\li\,:\,H_\Rl\,\hookrightarrow\, H$ is the
inclusion map, with differential $\u\li\,:\,{\mathfrak h}_\Rl\,\hookrightarrow\,{\mathfrak h}.$ The associated bundle-valued 1-forms satisfy
$$\w{\u\li\oc\lo}\,=\,\w\li\oc\w\lo\, ,$$
where the vector bundle map $\w\li\,:\,Q_\Rl\xx_{H_\Rl}{\mathfrak h}_\Rl\,\longrightarrow\, Q\xx_H{\mathfrak h}$ is defined by 
$$\w\li[p,\,\lb]\,=\,[\li p,\,\u\li\lb]$$ 
for all $p\,\in\, Q_\Rl\,\subset\, Q$ and $\lb\,\in\,{\mathfrak h}_\Rl\,\subset\,{\mathfrak h}.$
\end{proposition} 

If $H$ is a complex Lie group, then $Q\xx_H{\mathfrak h}$ is a complex vector bundle over $M.$ Let $(Q,\,\lk)$ be an involutive 
$H$-bundle. Given a connection $\lT$ on $Q$ we define the connection
$$\lT^\#\,:=\,\lk_*\lT$$
on $Q$ induced by the involution in \eqref{1}. We have ${\lT^\#}^\#\,=\,\lT$ and thus obtain a mapping $\lT\,\longmapsto
\,\lT^\#$ of order two on $\CL(Q).$ 

\begin{proposition} If $\lT^0\,=\,(\lT^0)^\#,$ then 
$$(\lT^\lo)^\#\,=\,\lT^{\u\#\oc\lo}$$
for any $\lo\,\in\,\lO^1_H(Q,\,{\mathfrak h}),$ where $\#$ is the involution on $H$, with differential $\u\#
\,=\,d_e\#$ on ${\mathfrak h}.$ The associated bundle-valued 1-forms satisfy
$$\w{\u\#\oc\lo}\,=\,\w\lk\oc\w\lo\, ,$$ 
where the (anti-linear) vector bundle map $\w\lk$ on $Q\xx_H{\mathfrak h}$ is defined by 
\begin{equation}\label{16}
\w\lk[q:\,\lb]\,:=\,[\lk q:\,\u\#\lb]
\end{equation}
for all $q\,\in\, Q$ and $\lb\,\in\,{\mathfrak h}.$ 
\end{proposition}

\begin{proposition}
A connection $\lT$ on an involutive principal $H$-bundle is hermitian, i.e.,
it is of the form $\lT\,=\,\li_*\lX$ for some connection 
$\lX$ on the fix point bundle $Q_\Rl\,\subset\, Q,$ if and only if $\lT^\#\,=\,\lT.$ Thus
$$\CL_\Rl(Q)\,=\,\li_*\CL(Q_\Rl)\,=\,\{\lT\in\CL(Q)\,\mid\ \lT^\#\,=\,\lT\}\, .$$ 
\end{proposition}

\begin{proof} Since $\lk\oc\li\,=\,\li$, it follows that
$$(\li_*\lX)^\#\,=\,\lk_*(\li_*\lX)\,=\,(\lk\oc\li)_*\lX\,=\,\li_*\lX\, .$$
Thus a hermitian connection $\lT$ satisfies the condition
$\lT^\#\,=\,\lT$.

Conversely, the condition $\lk_*\lT\,=\,\lT$ implies that $\lT\,=\,\lT^\lo$ for some 
$\lo\,\in\,\lO^1_H(Q,\,{\mathfrak h})$ satisfying $\lo\,=\,\u\#\oc\lo.$ Therefore, $\lo\,=\,\u\li\oc\lo_\Rl$ for
some $\lo_\Rl\,\in\,\lO^1_{H_\Rl}(Q_\Rl,\,{\mathfrak h}_\Rl).$ Setting $\lX\,:=\,\lX^{\lo_\Rl}$ we conclude
that $\li_*\lX\,=\,\lT$.
\end{proof}

We now introduce the central concept of this paper.

\begin{definition} Suppose that $M$ is a complex manifold and $H$ is a complex Lie group. We say that $\lO
\,\in\,\lO^2_H(Q,\,{\mathfrak h})$ is 
{\it pure} if for every $z\,\in\, M$, the associated bundle-valued 2-form $\w\lO_z$ is of type $(1,\,
1)$, i.e., its (unique) $\mathbb C$-bilinear extension $\w\lO_z\,:\,T_z^{\mathbb C} M\bigwedge
T_z^{\mathbb C} M\,
\longrightarrow\,(Q\xx_H{\mathfrak h})_z$ satisfies the condition
\begin{equation}\label{15}
\w\lO_z(v_1,\,v_2)\,=\,0\,=\,\w\lO_z(\overline{v}_1,\,\overline{v}_2)
\end{equation}
for all holomorphic tangent vectors $v_1,\,v_2\,\in\, T^{1,0}_zM$ and all anti-holomorphic tangent vectors
$\overline{v}_1,\,\overline{v}_2\,\in \,T^{0,1}_zM$. A connection $\lT$ on a principal
$H$-bundle $Q$ over $M$ is called {\it pure}, if its curvature form $\u\lT$ is pure. 
\end{definition}

Let
\begin{equation}\label{pc}
\y\CL(Q)\,\subset\,\CL(Q)
\end{equation}
denote the set of all pure connections on $Q.$ 

\begin{proposition} Let $(Q,\,\lk)$ be an involutive principal
$H$-bundle over a complex manifold $M.$ If a connection $\lT$ on $Q$ is pure, then 
the connection $\lT^\#$ is also pure.
\end{proposition}

\begin{proof} We have
$$\w{\u\lT}_z((d_q\lp)X,\,(d_q\lp)Y)\,=\,[q:\,\u\lT_q(X,Y)]$$
for all $X,\,Y\,\in\, T_qQ,$ where $z\,=\,\lp(q).$ For the connection $\lT^\#\,=\,\lk_*\lT$ the curvature satisfies
$$\lk^*\u\lT^\#\,=\,\lk^*\u{\lk_*\lT}\,=\,\u\#\oc\u\lT$$
according to \eqref{12}. This means that
$$\u\#(\u\lT_q(X,\,Y))\,=\,(\lk^*\u\lT^\#)_q(X,\,Y)\,=\,\u\lT^\#_{\lk q}((d_{\lk q}\lk)X,\,(d_{\lk q}\lk)Y).$$
Using the bundle map in \eqref{16} we have
$$\w\lk(\w{\u\lT}_z((d_q\lp)X,\,(d_q\lp)Y))\,=\,\w\lk[q:\,\u\lT_q(X,Y)]\,=\,[\lk q:\,\u\#(\u\lT_q(X,Y))]$$
$$=\,[\lk q:\,\u\lT^\#_{\lk q}((d_q\lk)X,\,(d_q\lk)Y)]\,=\, \w{\u\lT^\#}_z((d_{\lk q}\lp)(d_q\lk)X,\,(d_{\lk q}\lp)(d_q\lk)Y)
$$
$$
=\, \w{\u\lT^\#}_z((d_q\lp)X,\,(d_q\lp)Y)$$
because $\lp\oc\lk\,=\,\lp.$ Thus the equality
$$\w\lk(\w{\u\lT}_z(v_1,\,v_2))\,=\,\w{\u\lT^\#}_z(v_1,\,v_2)$$
holds for all $v_i\,\in\, T_zM.$ For the $\Cl$-bilinear extension we obtain the equality
$$\w{\u\lT^\#}_z(v_1,\,v_2)\,=\,\w\lk(\w{\u\lT}_z(\overline{v}_1,\,\overline{v}_2))$$
for all $v_i\,\in\, T_z^\Cl M,$ where $v\,\longmapsto\,\overline{v}$ is the canonical conjugation with fixed point
set $T_zM.$ Thus $\w{\u\lT^\#}_z$ satisfies the condition in \eqref{15} as well.
\end{proof}

It follows that there is a period 2 map $\lT\,\longmapsto\,\lT^\#$ on the space $\y\CL(Q)$
(see \eqref{pc}) of pure connections, 
whose fixed point set $\y\CL_\Rl(Q)$ consists of all pure hermitian connections on $Q.$

\section{Invariant connections}

For a $G$-equivariant principal bundle $Q$ the group $G$ acts on $\CL(Q)$ as follows: For any $g\,\in\, G$
the bundle map in \eqref{37}, associated with the identity map $id_H,$ yields the induced connection
$$g\cdot\lT\,:=\,(L_g^Q)_*\lT$$
satisfying $(L_g^Q)^*(g\cdot\lT)\,=\,\lT.$ Thus we have
\begin{equation}\label{17}
g\cdot\lT\,=\,(L^Q_{g^{-1}})^*\lT\, .
\end{equation} 
Let $\CL(Q)^G\,\subset\,\CL(Q)$ denote the set of all connections on $Q$ which are invariant under the
action of $G.$ For an equivariant hermitian structure $(Q,\,Q_\Rl)$ the group $G$ acts on $\CL(Q_\Rl)$ in
a manner analogous to \eqref{17}, and we let $\CL(Q_\Rl)^G\,\subset\,\CL(Q_\Rl)$ be the fixed point locus
for this action of $G$ on $\CL(Q_\Rl).$ The extension map $\lX\,\longmapsto
\,\li_*\lX$ commutes with the left action of $G$ on connections. It follows that there is a commuting diagram
$$\xymatrix{\CL^G(Q_\Rl)\ar[d]_{\li_*}&\subset &\CL(Q_\Rl)\ar[d]^{\li_*}\\ \CL^G(Q)&\subset &\CL(Q)}$$ 

\begin{proposition} If a connection $\lT^0$ is $G$-invariant, then 
$$g\cdot\lT^\lo\,=\,\lT^0+g\cdot\lo\,=\,\lT^{g\cdot\lo}$$
for all $\lo\,\in\,\lO^1_H(Q,\,{\mathfrak h}).$ If $(Q,\,Q_\Rl)$ is a hermitian structure, and $\lX^0$ is invariant, then 
$$g\cdot\lX^\lo\,=\,\lX^0+g\cdot\lo\,=\,\lX^{g\cdot\lo}$$
for all $\lo\,\in\,\lO^1_{H_\Rl}(Q_\Rl,\,{\mathfrak h}_\Rl).$ In both cases, the associated bundle-valued 1-forms satisfy
$$\w{g\cdot\lo}_{gx}\,=\,\w g\cdot (\w\lo_x\oc(d_{gx}L_{g^{-1}}^M))$$
using the action 
\begin{equation}\label{18}\w g\cdot [q,\,\lb]\,:=\,[gq,\,\lb]
\end{equation}
of $g\,\in\, G$ on $Q\xx_H{\mathfrak h}.$ (respectively, $Q\xx_{H_\Rl}{\mathfrak h}_\Rl.$) In particular, $\lT^\lo$ (respectively, $\lX^\lo$) is invariant if and only if
$$\w\lo_{gx}\oc(d_xL_g^M)\,=\,g\cdot\w\lo_x$$
for all $g\,\in\, G$ and $x\,\in\, M.$ 
\end{proposition}

\begin{proof} Since $\lT^\lo\,=\,\lT^0+\lo$ and $\lT^0$ is invariant, from \eqref{17} it follows
that $g\cdot\lo\,=\,(L_{g^{-1}}^Q)^*\lo.$ Thus we have
$$(g\cdot\lo)_{gq}\,=\,((L_{g^{-1}}^Q)^*\lo)_{gq}\,=\,\lo_q\oc(d_{gq}L_{g^{-1}}^Q)\, .$$
For every $X\,\in\, T_{gq}Q$, it follows that
$$\w{g\cdot\lo}_{gx}(d_{gq}\lp)X\,=\,[gq:\,(g\cdot\lo)_{gq}X]\,=\,[gq:\,\lo_q(d_{gq}L_{g^{-1}}^Q)X]$$
$$=\,g\cdot [q:\,\lo_q(d_{gq}L_{g^{-1}}^Q)X]\,=\,g\cdot\w\lo_x(d_q\lp)(d_{gq}L_{g^{-1}}^Q)X\,=\,g
\cdot\w\lo_x(d_xL_{g^{-1}}^M)(d_{gq}\lp)X$$
using the identity $\lp\oc L_{g^{-1}}^Q\,=\,L_{g^{-1}}^M\oc\lp.$ Thus
$$\w{g\cdot\lo}_{gx}v\,=\,g\cdot\w\lo_x(d_{gx}L_{g^{-1}}^M)v$$
for all $v\,\in\, T_{gx}M.$
\end{proof}

\begin{proposition}\label{ppc}
In case $G$ acts holomorphically on a complex manifold $M,$ the $G$-action on $\CL(Q)$ preserves the Hodge
types of the curvature. In particular, if $\lT$ is a pure connection, then $g\cdot\lT$ is also pure for
every $g\,\in\, G.$
\end{proposition}

\begin{proof} Since $\u{g\cdot\lT}\,=\,(L^Q_{g^{-1}})^*\u\lT$ we have
$$\u{g\cdot\lT}_{gq}(X,\,Y)\,=\,\u\lT_q((d_{gq}L_{g^{-1}}^Q)X,\,(d_{gq}L_{g^{-1}}^Q)X)$$ for all
$X,\,Y\,\in\, T_{gq}Q,$ and hence
$$\w{\u{g\cdot\lT}}_{gz}((d_{gq}\lp)X,\,(d_{gq}\lp)Y)\,=\,[gq:\,\u{g\cdot\lT}_{gq}(X,\,Y)]$$
$$=\,[gq:\,\u\lT_q((d_{gq}L_{g^{-1}}^Q)X,\,(d_{gq}L_{g^{-1}}^Q)Y)]\,=\,
g\cdot [q:\,\u\lT_q((d_{gq}L_{g^{-1}}^Q)X,\,(d_{gq}L_{g^{-1}}^Q)Y)]$$
$$=\,g\cdot\w{\u\lT}_z((d_q\lp)(d_{gq}L_{g^{-1}}^Q)X,\,(d_q\lp)(d_{gq}L_{g^{-1}}^Q)Y)
$$
$$
=g\cdot\w{\u\lT}_z((d_{gz}L_{g^{-1}}^M)(d_{gq}\lp)X,\,(d_{gz}L_{g^{-1}}^M)(d_{gq}\lp)Y)\, .$$
It follows that
$$\w{\u{g\cdot\lT}}_{gz}(v_1,\,v_2)\,=\,\w g\cdot\w{\u\lT}_z((d_{gz}L_{g^{-1}}^M)v_1,\,(d_{gz}L_{g^{-1}}^M)v_2)$$
for all $v_i\,\in\, T_{gz}M.$ Since the transformations $L_g^M$ are holomorphic, and the action
in \eqref{18} is $\Cl$-linear on the fibers, it follows that $\w{\u{g\cdot\lT}}$ has the same Hodge type as
that of $\w{\u\lT}.$ 
\end{proof}

{}From Proposition \ref{ppc} it follows that the action of $G$ preserves the subset $\y\CL(Q)$ in \eqref{pc}. Let
$$\y\CL(Q)^G\,=\,\y\CL(Q)\ui\CL(Q)^G$$
denote the set of all invariant pure connections on $Q.$ The involution $\lT\,\longmapsto\,\lT^\#$ on the
space $\y\CL(Q)^G$ of all pure invariant connections on $Q$ has the fixed point set $\y\CL_\Rl(Q)^G$ consisting of all pure invariant hermitian connections.

Now consider the special case where $M\,=\,G/K$ is a {\bf symmetric space}. Let ${\mathfrak g}$ (respectively, ${\mathfrak k}$) be the Lie algebra of $G$ 
(respectively, $K$). Both ${\mathfrak g}$ and ${\mathfrak k}$ are $K$-modules by the adjoint action. The Killing form on ${\mathfrak g}$ is non-degenerate. Let
$${\mathfrak p}\,=\,{\mathfrak k}^\perp \,\subset\, {\mathfrak g}$$
be the orthogonal complement of ${\mathfrak k}$ for the Killing form. Then the natural homomorphism
\begin{equation}\label{19}
{\mathfrak k}\oplus{\mathfrak p}\,\longrightarrow\, {\mathfrak g}
\end{equation}
is an isomorphism. The translates of ${\mathfrak p}$ by the left-translation action of $G$ on itself define a distribution $D
\,\subset\, TG$ which is in fact
preserved by the right-translation action of $K$ on $G$ because the decomposition
in \eqref{19} is $K$-equivariant. From this it follows immediately
that $D$ defines a connection, denoted by $\lG^0,$ on the principal $K$-bundle $\Gl\,:=\,G\xx_K K$. Since $D$ is
preserved by the left-translation action of $G$ on itself, we conclude that $\lG^0$ is invariant. Its connection
1-form coincides with the canonical projection
$$\lG^0_e\,:\,T_eG\,=\,{\mathfrak g}\,\longrightarrow\,{\mathfrak k}$$
at the identity element, and 
$$\lG^0_g\,:=\,\lG^0_e\oc(d_g L_{g^{-1}}^G)$$
for every $g\,\in\, G,$ using the left translations of $G$ on itself.

Consider the Lie bracket operation composed with the projection to ${\mathfrak k}$
\begin{equation}\label{20}
{\mathfrak p}\xt{\mathfrak p}\, \longrightarrow\, {\mathfrak g}\, \longrightarrow\, {\mathfrak k}\, .
\end{equation}
Left translations by elements of $G$ of this composition of homomorphisms defines a $\CL^\oo$ two-form
on $G/K$ with values in the adjoint vector bundle $\text{ad}({\mathbb G})$. This $\text{ad}({\mathbb G})$-valued
2-form is in fact the curvature of $\lG^0.$ 

\begin{proposition} For any homomorphism $f\,:\,K\,\longrightarrow\, H$ and $\lh\,\in\, H^f,$ the
{\bf tautological} connection
$$\lT^0\,:=\,(f_\lh)_*\lG^0$$ 
induced by the bundle map in \eqref{14} is invariant. If $f\,:\,K\,\longrightarrow\, H_\Rl$ and
$\lh\,\in\, H_\Rl$, then the induced connection
$$\lX^0\,:=\,(f_\lh)_*\lG^0$$
on $Q_\Rl\,=\,G\xx_{K,f}H_\Rl$ is invariant.
\end{proposition}

\begin{proof} Since $f_\lh$ intertwines the $G$-actions $L_g^G$ and $L_g^Q$ on $\Gl$ and $Q$ respectively, we have
$$f_\lh^*(L_g^Q)^*\lT^0\,=\,(L_g^Q\oc f_\lh)^*\lT^0\,=\,(f_\lh\oc L_g^G)^*\lT^0\,=\,(L_g^G)^*f_\lh^*\lT^0\,
=\,(L_g^G)^*\lG^0\,=\,\lG^0\,=\,f_\lh^*\lT^0$$
for the respective connection 1-forms. It follows that $(L_g^Q)^*\lT^0\,=\,\lT^0$ for all $g\,\in\, G.$ More
directly, we can argue that
$$g\cdot\lT^0\,=\,(L_g^Q)_*\lT^0\,=\,(L_g^Q)_*(f_\lh)_*\lG^0)\,=\,(L_g^Q\oc f_\lh)_*\lG^0
$$
$$
=\,(f_\lh\oc L_g^G)_*\lG^0\,=\,(f_\lh)_*((L_g^G)_*\lG^0)\,=\,(f_\lh)_*\lG^0\,=\,\lT^0\, .$$
This proves the proposition.
\end{proof}

In \cite[Section 3]{BU} the connection $\lT^0$ (for $\lh\,=\,e_H$) is described more explicitly.

A real-linear map $\cl\,:\,T_oM\,\longrightarrow\,{\mathfrak h}$ is called {\bf $f$-covariant} if
\begin{equation}\label{23}
\cl\oc(d_oL_k^M)\,=\,Ad_{f(k)}^H\oc\cl
\end{equation}
for all $k\,\in\, K.$ Let $\LL(T_oM,\,{\mathfrak h})^f$ denote the vector space of all $f$-covariant linear mappings. If $f\,:\,K
\,\longrightarrow\, H_\Rl,$ we define $\LL(T_oM,\,{\mathfrak h}_\Rl)^f$ in a similar way. 

\begin{proposition}\label{a} There is a 1-1 correspondence
\begin{equation}\label{24}
\LL(T_oM,\,{\mathfrak h})^f\,\longrightarrow\,\lO^1_H(Q,\,{\mathfrak h})^G,\ \ \cl\,\longmapsto\,\lo
\end{equation}
which is uniquely determined by
\begin{equation}\label{21}
\lo_\ql X\,=\,\cl((d_\ql\lp)X)
\end{equation}
for all $X\,\in\, T_\ql Q,$ where $\ql\,\in\, Q_o$ satisfies \eqref{6}. The associated bundle-valued 1-form is given by
\begin{equation}\label{22}\w\lo_xv\,:=\,[g\ql,\,\cl(d_xL_{g^{-1}}^M)v]
\end{equation}
for all $v\in T_xM,$ where $x\,=\,g(o)$ for some $g\,\in\, G.$ For a hermitian structure $(Q,\,Q_\Rl)$
there is a similar 1-1 correspondence
\begin{equation}\label{25}
\LL(T_oM,\,{\mathfrak h}_\Rl)^f\,\longrightarrow\,\lO^1_{H_\Rl}(Q_\Rl,\,{\mathfrak h}_\Rl)^G,\ \ \cl\,\longmapsto\,\lo
\end{equation}
such that \eqref{21} and \eqref{22} hold for any $\ql\,\in\, (Q_\Rl)_o.$ 
\end{proposition}

\begin{proof} Since
$$[gk\ql,\,\cl(d_xL_{(gk)^{-1}}^M)v]\,=\,[g\ql f(k),\,\cl(d_oL_{k^{-1}}^M)(d_xL_{g^{-1}}^M)v]
\,=\,[g\ql,\, Ad_{f(k)}^H\cl(d_oL_{k^{-1}}^M)(d_xL_{g^{-1}}^M)v]$$ 
for all $k\,\in\, K,$ the covariance condition in \eqref{23} ensures that $\w\lo_x$ is well-defined, i.e.,
it is independent of the choice of $g.$
Conversely, every $G$-invariant 1-form of type $Ad_H$ on $M$ arises this way.
\end{proof}

The group $H^f$ defined in \eqref{9} acts on $\LL(T_oM,\,{\mathfrak h})^f$ via $\cl\,\longmapsto\, I^H_\lh\oc\cl,$ for $\lh\,
\in \,H^f.$ Similarly, for $f\,:\,K\,\longrightarrow\, H_\Rl,$ the group $H_\Rl^f$ acts on $\LL(T_oM,\,{\mathfrak h}_\Rl)^f.$

The next result is a sharpening of \cite[Proposition 3.3]{BU}.

\begin{theorem}\label{b} For a fixed homomorphism $f\,:\,K\,\longrightarrow\, H,$ the assignment
\begin{equation}\label{26}
\cl\,\longmapsto\,\lT^\cl\,:=\,\lT^\lo\,=\,\lT^0+\lo
\end{equation}
described in \eqref{24} yields a canonical bijection
$$\LL(T_oM,\,{\mathfrak h})^f/H^f\, \stackrel{\sim}{\longrightarrow}\, \CL(Q)^G$$ 
which classifies all invariant connections on $Q\,=\,G\xx_{K,f}H.$ If $f\,:\,K\,\longrightarrow\, H_\Rl,$ the assignment
$$\cl\,\longmapsto\,\lX^\cl\,:=\,\lX^\lo\,=\,\lX^0+\lo$$ 
specified in \eqref{25} yields a canonical bijection
$$\LL(T_oM,\,{\mathfrak h}_\Rl)^f/H_\Rl^f\, \stackrel{\sim}{\longrightarrow}\,\CL(Q_\Rl)^G$$ 
which classifies all invariant connections on $Q_\Rl\,=\,G\xx_{K,f}H_\Rl.$
\end{theorem}

\begin{proof}
Using Proposition \ref{a} and \cite[Proposition 3.3]{BU} it can be shown that every invariant connection $\lT$ on $Q$
is of the form $\lT^\cl$ for some $\cl\,\in\,\LL(T_oM,\,{\mathfrak h})^f.$ However, the assignment in
\eqref{26} is not injective, because the 
tautological connection $\lT^0$ depends on the choice of base point $\ql\,\in\, Q_o$ satisfying \eqref{6}. For any 
$\lh\,\in\, H^f$ we may choose $\ql'\,=\,\ql\lh$ as an equivalent base point, which leads to a different choice of $\lT^0$ 
and accordingly, to another realization of $\cl.$ It is easy to see that now $\cl$ is replaced by 
$\cl':=I^H_\lh\oc\cl.$
\end{proof}

\section{Hermitian symmetric spaces}

Now suppose that $M\,=\,G/K$ is a hermitian symmetric space. Then the complexified tangent space
$T_o^\Cl$ has a decomposition
$$T_o^\Cl M\,=\,Z\xx\overline{Z}\, ,$$
where $Z\,=\,T_o^{1,0}M$ carries the structure of a hermitian Jordan triple, and
$\overline{Z}\,=\,T_o^{0,1}M.$ For background on Jordan triples and the Jordan theoretic realization of bounded symmetric
domains, we refer to \cite{FK,L,U}. The `real' tangent space is recovered as the diagonal
$$T_oM\,=\,\{v+\overline{v}\,\mid\ v\,\in\, Z\}\, .$$
In this realization, we have
$$(d_oL_k^M)(v+\overline{v})\,=\,kv+\overline{kv}$$
for all $k\,\in\, K$; this way we obtain the entire
identity component of the Jordan triple automorphism group of $Z.$ Every real linear map 
$\cl\,:\,T_oM\,\longrightarrow\,{\mathfrak h}$ has a unique decomposition
$$\cl(v+\overline{v})\,=\,\al v+\bl\overline{v}\, ,$$
where $\al\,:\,Z\,\longrightarrow\,{\mathfrak h}$ and $\bl\,:\,\overline{Z}\,\longrightarrow\,{\mathfrak h}$ are linear
(respectively, anti-linear) mappings. For $f\,:\,K
\,\longrightarrow\, H$, if $\cl$ is $f$-covariant, the identity
$$\al(kv)+\bl(\overline{kv})\,=\,\cl((d_oL_k^M)(v+\overline{v}))\,=\,Ad_{f(k)}^H\cl(v+\overline{v})\,=\,Ad_{f(k)}^H\al(v)+Ad_{f(k)}^H\bl(\overline{v})$$
shows that $\al$ and $\bl$ are $f$-covariant in the sense that
$$\al(kv)\,=\,Ad_{f(k)}^H\al(v),\quad\bl(\overline{kv})\,=\,Ad_{f(k)}^H\bl(\overline{v})$$
for all $k\,\in\, K.$ Here we use that $L_k^M$ is holomorphic and $Ad_{f(k)}^H$ is $\Cl$-linear. We write $\cl
\,=\,(\al,\,\bl)$ and
$\lT^\cl\,=\,\lT^{\al,\bl}.$ Let $\LL(Z,\,{\mathfrak h})^f$ (respectively, $\LL(\overline{Z},\,{\mathfrak h})^f$) denote the $\Cl$-vector space of
all $\Cl$-linear (respectively, anti-linear) mappings which are $f$-covariant. The group $H^f$ defined in
\eqref{9} acts on both $\LL(Z,\,{\mathfrak h})^f$ and $\LL(\overline{Z},\,{\mathfrak h})^f.$ If we have $f\,:\,K\,\longrightarrow\, H_\Rl$, then the
involution $\lT\,\longmapsto\,\lT^\#$ is realized as
$$(\lT^{\al,\bl})^\#\,=\,\lT^{\u\#\oc\bl,\u\#\oc\al}\, .$$
Moreover, the $f$-covariant maps 
$\cl\,:\,T_oM\,\longrightarrow\,{\mathfrak h}_\Rl$ have the form
$$\cl(v+\overline{v})\,=\,\al(v)+\u\#\al(v)\, ,$$ 
where $a\,\in\,\LL(Z,\,{\mathfrak h})^f.$ 

Applying Theorem \ref{b} we obtain the following:

\begin{theorem}\label{f} Let $M\,=\,G/K$ be hermitian symmetric. For any homomorphism $f\,:\,K
\,\longrightarrow\, H$ there is a canonical bijection
$$\big(\LL(Z,\,{\mathfrak h})^f\xx\LL(\overline{Z},\,{\mathfrak h})^f\big)/H^f\, \stackrel{\sim}{\longrightarrow}\,
\CL(Q)^G,\quad(\al,\,\bl)\,\longmapsto\,\lT^{\al,\bl},$$
which classifies all invariant connections on the homogeneous $H$-bundle $Q\,=\,G\xx_{K,f}H.$ If $f\,:\,K
\,\longrightarrow\, H_\Rl,$ there is a canonical bijection
$$\LL(Z,\,{\mathfrak h})^f/H^f\, \stackrel{\sim}{\longrightarrow}\,\CL_\Rl(Q)^G,\quad\al\,\longmapsto\,\lT^{\al,\#\al}\, ,$$
which classifies all invariant hermitian connections on the hermitian structure $(Q,\, Q_\Rl).$
\end{theorem}

\begin{corollary}\label{i} For any hermitian symmetric space $M\,=\,G/K,$ the space $\CL(Q)^G$ is a normal
variety, such that the involution $\lT\,\longmapsto\,\lT^\#$ is anti-holomorphic.
\end{corollary}

\begin{proof} By Theorem \ref{f},
$$\CL(Q)^G\,=\,\big(\LL(Z,\,{\mathfrak h})^f\xx\LL(\overline{Z},\,{\mathfrak h})^f\big)/H^f$$
is a quotient variety of a linear space under the action of an algebraic subgroup $H^f$ of $H.$ By 
\cite[Theorem 16.1.1 and Lemma 16.1.2]{BB}, this is a normal variety.
\end{proof}

Now we turn to the classification of {\bf pure} connections. The group $G$ acts by holomorphic automorphisms $L_g^M$ on the hermitian symmetric space $M.$ The complexification of ${\mathfrak p}$ has a decomposition
$${\mathfrak p}^\Cl\,:=\,{\mathfrak p}\xt_\Rl \Cl\,=\, {\mathfrak p}_+\oplus{\mathfrak p}_-\, .$$
This decomposition produces the Hodge type decomposition of the complexified tangent bundle
$$T(G/K)\xt_\Rl\Cl\,=\, T^{1,0}(G/K)\oplus T^{0,1}(G/K)\, .$$
The complexification of the composition in \eqref{20} vanishes on
both ${\mathfrak p}_+\xt{\mathfrak p}_+$ and ${\mathfrak p}_-\xt{\mathfrak p}_-$. Consequently, both the 
$(2,\,0)$ and $(0,\,2)$ type components of the curvature of $\lG^0$ vanish. Therefore,
the curvature of $\lG^0$ is of Hodge type $(1,\,1)$, showing that the connection $\lG^0,$ extended to $G\xx_K K^\Cl$ via the embedding $K\,\subset\, K^\Cl,$ is pure. This implies that the induced tautological connection $\lT^0$ on $Q$ is also pure. In more detail:

\begin{proposition}\label{c}
Let $M\,=\,G/K$ be a hermitian symmetric space. For any homomorphism $f\,:\,K
\,\longrightarrow\, H$ and any $\lh\,\in\, H^f,$ the tautological connection $\lT^0\,=\,(f_\lh)_*\lG^0$ on $Q$ is pure.

If $f\,:\,K\,\longrightarrow\, H_\Rl$, and $\lh\,\in\, H_\Rl^f,$ then the tautological connection 
$\lX^0\,=\,(f_\lh)_*\lG^0$ on $Q_\Rl$ has a pure hermitian extension $\li_*\lX^0.$
\end{proposition}

\begin{proof}
By \cite[Proposition II.6.2]{KN} the connection 1-form and curvature 2-form are related by
$$f_\lh^*\lT^0\,=\,f_\lh^*((f_\lh)_*\lG^0)\,=\,\u f\oc\lG^0\, ,$$
$$f_\lh^*\u\lT^0\,=\,f_\lh^*\u{(f_\lh)_*\lG^0}\,=\,\u f\oc\u\lG^0\, .$$
Thus we have a commuting diagram
\begin{equation}\label{27}
\xymatrix{\bigwedge^2 T_gG\ar[r]^{d_gf_\lh}\ar[d]_{\u\lG^0_g}&\bigwedge^2
T_qQ\ar[d]^{\u\lT^0_q}\\ {\mathfrak k}\ar[r]_{\u f}&{\mathfrak h}}
\end{equation}
where $g\,\in\, G$ and $q\,=\,f_\lh(g).$ Consider the induced vector bundle
$$G\xx_K{\mathfrak k}\,=\,\{[g:\,\la]\,=\,[gk:\,Ad^K_{k^{-1}}\la]\,\mid\ g\,\in\, G,\ \la\,\in\,{\mathfrak k},\ k\,\in\, K\}\, .$$
Then $f_\lh$ induces a vector bundle homomorphism $\t f_\lh\,:\,G\xx_K{\mathfrak k}\,\longrightarrow\, Q\xx_H{\mathfrak h}$ over
$M,$ defined by
$$\t f_\lh[g,\,\la]\,:=\,[f_\lh(g),\,\u f\la]\, .$$
Now let $v_1,\,v_2\,\in \,T_zM,$ where $z\,=\,g(o).$ Choose $w_i\,\in\, T_gG$ with $(d_gR^M_o)w_i
\,=\,v_i.$ Define $\t v_i\,:=\,(d_g f_\lh)w_i.$ Then the commuting diagram in \eqref{27} implies that
$$(d_q\lp)\t v_i\,=\,(d_q\lp)(d_g f_\lh)w_i\,=\,(d_gR^M_o)w_i\,=\,v_i\, .$$
It now follows that
$$\w{\u\lT}^0_z(v_1,\,v_2)\,=\,[q:\,\u\lT^0_q(\t v_1,\,\t v_2)]\,=\,[q:\,\u\lT_q^0((d_g f_\lh)w_1,\,(d_g f_\lh)w_2)]$$
$$=\, [q:\,(f_\lh^*\u\lT^0)_g(w_1,\,w_2)]\,=\,[q:\,\u f(\u\lG^0_g(w_1,\,w_2))]
\,=\,\t f_\lh\big(\w{\u\lG}^0_z(v_1,\,v_2)\big)\, .$$

Consider the complexified principal bundle $G\xx_K{\mathfrak k}^\Cl$ and the map $\t f_\lh\,:\,G\xx_K{\mathfrak k}^\Cl
\,\longrightarrow\, Q\xx_H{\mathfrak h}$ which is $\Cl$-linear on the fibers.
Passing to the complexification of $\u\lG^0,$ it follows that the relation
$$\w{\u\lT}^0_z(v_1,\,v_2)\,=\,\t f_\lh(\w{\u\lG}^0_z(v_1,\,v_2))$$
also holds for complexified tangent vectors $v_i\,\in\, T_z^\Cl M.$ There exist $u_i
\,\in\, T_oM$ such that $v_i\,=\,T_o(L_g^M)u_i.$ Since $\lG^0$ is invariant, we have
$$\w{\u\lG}^0_z(v_1,\,v_2)\,=\,\w{\u\lG}^0_{g(o)}((d_oL_g^M)u_1,\,(d_oL_g^M)u_2)\,=\,
\w{\u\lG}^0_o(u_1,u_2)\,=\,0$$
because the constant vector fields 
$$u_i\f{\dl}{\dl z}\,\in\,{\mathfrak g}^\Cl$$ 
commute. Similarly, we have
$$\w{\u\lT}^0_z(\overline{v}_1,\,\overline{v}_2)\,=\,\t f_\lh(\w{\u\lG}^0_z(\overline{v}_1,\,\overline{v}_2))$$
for anti-holomorphic tangent vectors, and 
$$\w{\u\lG}^0_z(\overline{v}_1,\,\overline{v}_2)\,=\,\w{\u\lG}^0_{go}(d_oL_g^M)\overline{u}_1,\ \
(d_oL_g^M)\overline{u}_2)\,=\,\w{\u\lG}^0_o(\overline{u}_1,\,\overline{u}_2)\,=\,0$$
since the quadratic vector fields 
$$\overline{u}_i\,=\,\{z;u_i;z\}\f{\dl}{\dl z}\,\in\,{\mathfrak g}^\Cl$$ 
commute as a consequence of the Jordan triple identity.
\end{proof}

The following theorem is our first main result.

\begin{theorem}\label{g}
Let $M\,=\,G/K$ be a hermitian symmetric space. Consider the homogeneous complex
principal $H$-bundle $Q\,=\,G\xx_{K,f}H$ induced by the homomorphism $f\,:\,K\,\longrightarrow\, H.$ Then
an invariant connection $\lT^\cl\,=\,\lT^{\al,\bl},$ associated with an $f$-covariant (real) linear map
$$\cl\,=\,\al+\bl\,:\,Z_\Rl\,\longrightarrow\,{\mathfrak h}$$ is pure if and only if
\begin{equation}\label{29}
\al\wedge\al\,=\,0\,=\, \bl\wedge\bl\, .
\end{equation}
\end{theorem}

\begin{proof} Write $\lT^\cl\,=\,\lT^0+\lo.$ Then \cite[Theorem II.5.2 and Proposition II.5.5]{KN} imply that
$$\u\lT^\cl_q(X,\,Y)\,=\,(d\lT^\cl)_q(X,\,Y)+\f12[\lT^\cl_q X,\,\lT^\cl_q Y]\,=\,(d\lT^0+d\lo)_q(X,\,Y)
$$
$$
+\f12[(\lT^0_q+\lo_q)X,\,(\lT^0_q+\lo_q)Y]
\,=\,(d\lT^0)_q(X,\,Y)+\f12[\lT^0_qX,\,\lT^0_qY]+(d\lo)_q(X,\,Y)
$$
$$
+\f12[\lT^0_qX,\, \lo_qY]+\f12[\lo_qX,\, \lT^0_qY]+\f12[\lo_qX,\, \lo_qY]$$
$$=\,\u \lT^0_q(X,\,Y)+(D^0\lo)_q(X,\,Y)+(\lo\wedge\lo)_q(X,\,Y)$$
for all $X,\,Y\,\in\, T_qQ,$ where $D^0\lo$ is the covariant derivative of the tensorial 1-form $\lo$ with respect 
to $\lT^0.$ By Proposition \ref{c}, the curvature $\u\lT^0$ is of type $(1,\,1)$. Thus the main point of the proof is 
to show that
\begin{equation}\label{28}
D^0\lo\,=\,0\, .
\end{equation}

The proof of \eqref{28}
presented here is a simplification of a similar argument in \cite[Theorem 5.2]{BU}. For
any $\lg\,\in\,{\mathfrak g}$ consider the induced vector field 
$$\lg_x^M\,=\,\f{\dl}{\dl t}\big|_{t=0}(\exp(t\lg)x)\,=\,\f{\dl}{\dl t}\big|_{t=0}(R_x^M\exp(t\lg))\,
=\,(d_eR_x^M)\lg$$
on $M$ induced by the left $G$-action. Let $\lg^Q$ denote the $\lT^0$-horizontal lift of $\lg^M.$ 

The following two lemmas would be needed to prove \eqref{28}.

\begin{lemma}\label{d} For any $g\,\in\, G$ and $q\,=\,g\ql\,=\,\t f(g)$, the following holds:
$$(\lo\lg^Q)(\t f(g))\,=\,\cl((d_eR_o^M)Ad^G_{g^{-1}}\lg)\, .$$
\end{lemma}

\begin{proof}
Put $z\,=\,g(o)\,=\,\lp(q).$ The identity $L_{g^{-1}}^M\oc R_z^M\,=\,R_o^M\oc I^G_{g^{-1}}$ implies 
that $$(d_zL_{g^{-1}}^M)\lg_z\,=\,(d_zL_{g^{-1}}^M)(d_eR_z^M)\lg\,=\,(d_eR_o^M)Ad^G_{g^{-1}}\lg$$ and hence
$$[q:\,(\lo\lg^Q)(q)]\,=\,[q:\,\lo_q\lg^Q_q]\,=\,\w\lo_z(d_q\lp)\lg^Q_q
\,=\,\w\lo_z\lg_x\,=\,[q:\,\cl(d_zL_{g^{-1}}^M)\lg_z]
$$
$=\, [q:\,\cl(d_eR_o^M)Ad^G_{g^{-1}}\lg]$.
\end{proof}

\begin{lemma}\label{e} For any $\lg,\,\lh\,\in\, {\mathfrak g}$, the following holds:
$$(\lg^Q(\lo\lh^Q))(\ql)\,=\,-\cl(d_eR_o^M)[\lg,\,\lh]\, .$$
\end{lemma} 

\begin{proof} Let $\lg\,\in\,{\mathfrak p}.$ Then
$$(d_\ql\lp)(d_e\t f)\lg\,=\,(d_e(\lp\oc\t f))\lg\,=\,(d_eR_o^M)\lg\,=\,\lg_o^M\,=\,(d_\ql\lp)\lg^Q_\ql$$ and
$(d_e\t f)\lg\,\in\,(d_e\t f){\mathfrak p}\,=\,(d_e\t f)T_e^0G\,=\, T_\ql^0Q.$ It follows that
$$\lg^Q_\ql\,=\,(d_e\t f)\lg\, .$$
Let $g_t\,:=\,\exp(t\lg)\,\in \,G.$ Since $d_eR_o^M$ is linear, Lemma \ref{d} implies that 
$$(\lg^Q(\lo\lh^Q))(\ql)\,=\,(d_\ql(\lo\lh^Q))\lg^Q_\ql\,=\,(d_\ql(\lo\lh^Q))(d_e\t f)\lg
\,=\,\f{\dl}{\dl t}\big|_{t=0}(\lo\lh^Q)\t f(g_t)$$
$$=\,\f{\dl}{\dl t}\big|_{t=0}\cl((d_eR_o^M)Ad^G_{g_t^{-1}}\lh)\,=\,
\cl((d_eR_o^M)\f{\dl}{\dl t}\big|_{t=0}Ad^G_{g_t^{-1}}\lh)
\,=\,-\cl(d_eR_o^M)[\lg,\, \lh]\, .$$
This completes the proof.
\end{proof}

We can now conclude the proof of \eqref{28}. For $u\,\in\, T_oM$, consider the vector field 
$$\d u\,=\,u+\Le\{z;u;z\}\,\in\,\d {\mathfrak p}$$
on $M.$ Then
$$(D^0\lo)(\d u^Q,\,\d v^Q)\,=\,d\lo(\d u^Q,\,\d v^Q)\,=\,\d u^Q(\lo\d v^Q)-\d v^Q(\lo\d u^Q)-\lo[\d u^Q,\d v^Q]$$
and hence
$$(D^0\lo)(\d u^Q,\d v^Q)(\ql)\,=\,\d u^Q(\lo\d v^Q)(\ql)-\d v^Q(\lo\d u^Q)(\ql)-\lo_\ql[\d u^Q,\,\d v^Q]_\ql\, .$$
By Lemma \ref{e} we have 
$$\d u^Q(\lo\d v^Q)(\ql)\,=\,-\cl(d_eR_o^M)[\d u,\,\d v]\,=\,0$$
since $[\d u,\,\d v]\,\in\,{\mathfrak k}$, and $R_o^Mk\,=\,ko\,=\,o$ implies
that $(d_eR_o^M){\mathfrak k}\,=\,0.$ Similarly, we have $\d v^Q(\lo\d u^Q)(\ql)\,=\,0.$

By definition, we have $(d_q\lp)\d u^Q_q\,=\,\d u_{\lp(q)}.$ Now applying \cite[Proposition I.3.3]{H} it follows that 
$(d_\ql\lp)[\d u^Q,\,\d v^Q]_\ql\,=\,[\d u,\,\d v]_o\,=\,0.$ Hence $\lo_\ql[\d u^Q,\,\d v^Q]_\ql\,=\,0$,
 and therefore $$(D^0\lo)_\ql(\d u^Q_\ql,\,\d v^Q_\ql)\,=\,0.$$
This proves \eqref{28} because $D^0\lo$ is $G$-invariant. It follows that pure invariant connections are characterized by the condition that 
$\cl\bigwedge\cl$ is of type $(1,\,1)$. This is equivalent to \eqref{29}.
\end{proof}

A shorter proof of \eqref{28} (and hence of the Theorem \ref{g}) uses the fact that $D^0\lo$ is a $G$-invariant
section of $\text{ad}(Q)\xt(\bigwedge^2 T^*_\Cl(G/K))$. Therefore, to prove \eqref{28} it suffices to show that
\begin{equation}\label{30}(D^0\lo)_{eK}\, =\, 0\, .
\end{equation}

Let $\sqrt{-1}E\,=\, \sqrt{-1}z\f{\dl}{\dl z}$ generate the center of ${\mathfrak k}.$ Put
$${\mathfrak h}_\pm\,:=\,\{\lb\,\in\,{\mathfrak h}\,\mid\ [\u f(\sqrt{-1}E),\,\lb]\,=\,\pm \sqrt{-1}\lb\}\, .$$
Then $(D^0\lo)_{eK}$ is an element of $({\mathfrak h}_+\oplus{\mathfrak h}_-)\xt\bigwedge^2{\mathfrak p}^\Cl$. Now the center
$\Tl$ of $K$ acts on 
$\bigwedge^2{\mathfrak p}^\Cl$ with weights $2,\,0,\,-2$, while $\Tl$ acts on ${\mathfrak h}_+\oplus {\mathfrak h}_-$ with
weights $1,\,-1$. Therefore, there is no nonzero 
$\Tl$-invariant element in $({\mathfrak h}_+\oplus{\mathfrak h}_-)\xt\bigwedge^2{\mathfrak p}^\Cl.$ This proves \eqref{30}. Hence
\eqref{28} holds because $D^0\lo$ is $G$-invariant.

Put
$$\y\LL(Z,\,{\mathfrak h})^f\,:=\,\{\al\,\in\,\LL(Z,\,{\mathfrak h})^f\,\mid\ \al\wedge\al\,=\,0\}$$
and define $\y\LL(\overline{Z},\,{\mathfrak h})^f$ in a similar way. The group $H^f$ defined in \eqref{9} acts on both $\y\LL(Z,\,{\mathfrak h})^f$ and 
$\y\LL(\overline{Z},\,{\mathfrak h})^f.$

Combining Theorem \ref{f} and Theorem \ref{g} yields the following:

\begin{theorem}\label{h} For any homomorphism $f\,:\,K\,\longrightarrow\, H$ there is a
canonical bijection
$$\big(\y\LL(Z,\,{\mathfrak h})^f\xx\y\LL(\overline{Z},\,{\mathfrak h})^f\big)/H^f\,\stackrel{\sim}{\longrightarrow}
\, \y\CL(Q)^G,\quad(\al,\, \bl)\,\longmapsto\,\lT^{\al,\bl},$$
which classifies all invariant pure connections on the homogeneous
principal $H$-bundle $Q\,=\,G\xx_{K,f}H.$

If $f\,:\,K
\,\longrightarrow\, H_\Rl,$ there is a canonical bijection
$$\y\LL(Z,\,{\mathfrak h})^f/H^f \,\stackrel{\sim}{\longrightarrow}\,
\y\CL_\Rl(Q)^G,\quad\al\,\longmapsto\,\lT^{\al,\u\#\oc\al}\, ,$$
which classifies all invariant pure hermitian connections on the hermitian structure $(Q,\,Q_\Rl).$ 
\end{theorem}

In terms of the preceding theorem, the involution $\lT\,\longmapsto\,\lT^\#$ acts on the space 
$\y\CL(Q)^G,$ with fixed point set $\y\CL_\Rl(Q)^G$ consisting of all pure invariant hermitian 
connections.

\section{Fine structure of moduli spaces}

We will now analyze the structure of the moduli space $\y\CL(Q)^G$ of all pure invariant 
connections on a homogeneous $H$-bundle $Q\,=\,G\xx_{K,f}H$ in more detail. The goal is to show that 
$\y\CL(Q)^G$ is a {\bf normal complex variety} on which the canonical involution 
$\lT\,\longmapsto\,\lT^\#$ is anti-holomorphic. According to Theorem \ref{h} we start with a 
homomorphism $f\,:\,K\,\longrightarrow\, H$ and two $f$-covariant linear maps 
$\al\,:\,Z\,\longrightarrow\,{\mathfrak h},\ \bl \,:\,\overline{Z}\,\longrightarrow\,{\mathfrak h},$ where $Z$ is a hermitian 
Jordan triple identified with the holomorphic tangent space $T_o^{1,0}M$ of the hermitian 
symmetric space $M\,=\,G/K$ at the origin. We may assume that $Z$ is irreducible of rank $r.$ Let 
$e_1,\,\cdots, \,e_r$ be a {\bf frame} of minimal orthogonal tripotents $e_i\,\in\, Z.$ By the {\bf 
spectral theorem} for Jordan triples \cite{L}, every $z\,\in\, Z$ has a representation
$$z\,=\,\sum_i t_i\ ke_i\, $$
where $k\,\in \,K$ and the (uniquely determined) ``singular values'' satisfy
$0\,\le \,t_1\,\le\, t_2\le\,\cdots\,\le\, t_r.$ For each $i$ we put
$$K_i\,:=\,\{k\,\in\, K\,\mid\ ke_i\,=\,e_i\}\, .$$
These subgroups of $K$ are all conjugate.

\begin{lemma}\label{j}
Let $\al\,\in\,\LL(Z,\,{\mathfrak h})^f$ and $\bl\,\in\,\LL(\overline{Z},\,{\mathfrak h})^f.$ Then $\al\bigwedge\al\,=\,0$
(respectively, $\bl\bigwedge\bl\,=\,0$) if and only if 
$[\al(e_i),\,\al(e_j)]\,=\,0$ (respectively, $[\bl(e_i),\,\bl(e_j)]\,=\,0$)
for all $1\,\le\, i,\,j\,le \,r.$
\end{lemma}

\begin{proof} Working out the first case, the condition that $\al\bigwedge\al\,=\,0$ is equivalent to 
the condition that
\begin{equation}\label{31}
[\al(z),\,\al(w)]\,=\,0
\end{equation} 
for all $z,\,w\,\in\, Z.$ Now assume that $[\al(e_i),\,\al(e_j)]\,=\,0$ for all
$1\,\le\, i,\,j\,\le\, r.$ Assume that $k\in K$ satisfies the condition
$ke_j\,=\,e_j$ for some fixed $j.$ Then we have
$Ad^H_{f(k)}\al(e_j)\,=\,\al(e_j).$ In view of \eqref{1} we obtain the equality
\begin{equation}\label{32}
[\al(z),\,\al(e_j)]\,=\,\sum_i\lz_i Ad^H_{f(k)}[\al(e_i),\,\al(e_j)]\,=\,0\, .
\end{equation}
For any fixed $j$ the orbit $\{\sum_i\lz_i\ ke_i\,\mid\ \lz_i\,\in\,\Cl,\ k\,\in\, K,\ ke_j
\,=\,e_j\}$ has a non-empty interior. Hence from \eqref{32} it follows that
$[\al(z),\,\al(e_j)]\,=\,0$ for all $z\,\in\, Z.$ Using again the spectral theorem
for Jordan triples we conclude that \eqref{31} holds.
\end{proof}

\begin{proposition}
For rank $r\,=\,1,$ corresponding to the unit ball in $G/K\,=\, \Cl^d$ (non-compact version) and the projective 
space $G/K\,=\,\Cl{\mathbb P}^d$ (compact version), every invariant connection $\lT$ on $Q$ is pure. Hence, by 
Corollary \ref{i}, $\y\CL(Q)^G\,=\,\CL(Q)^G$ is a normal variety.
\end{proposition}

\begin{proof}
For $Z\,=\,\Cl$ (corresponding to the unit disk and the Riemann sphere, respectively) the condition 
$\al\bigwedge\al\,=\,0\,=\,\bl\bigwedge\bl$ is trivially satisfied since $Z\bigwedge Z\,=\,0.$ In the
higher-dimensional rank 1 case, we have only one tripotent $e\,=\,e_1,$, and Lemma \ref{j} shows that the condition 
$\al\bigwedge\al\,=\,0\,=\,\bl\bigwedge\bl$ is again satisfied.
\end{proof}

In the 1-dimensional case, we can describe the moduli space very explicitly. Consider the case
where $Z\,=\,\Cl,\,K\,=\,{\rm U}(1)\,=\,\Tl.$ Let 
$H\,=\,{\rm GL}_N(\Cl),\,H_\Rl\,=\,{\rm U}(N).$ The vector space $V\,=\,\Cl^N$ has a (finite) direct sum decomposition 
\begin{equation}\label{33}
V\,=\,\sum_{m\in\Zl}V_m\, ,
\end{equation}
where
$$V_m\,=\,\{v\,\in\, V\,\mid\ f(\lt)v\,=\,\lt^m v\ \ \forall\ \lt\,\in\,\Tl\}$$
is the weight space with respect to the action of $\Tl.$ Up to a permutation we order the weight spaces according
to strictly increasing weights $m_0\,<\,m_1\,<\,\cdots\, < \,m_\l.$ Any matrix in $A\,\in\,{\mathfrak h}
\,=\,\mathfrak{gl}_N(\Cl)$ can be written as a block matrix $A=(A_i^j)$ with 
$A_i^j\,:\,V_{m_j}\,\longrightarrow\, V_{m_i}.$ Since
$$f(\lt)\,=\,\cc{\lt^{m_0}I_{m_0}}000{\ddots}000{\lt^{m_\l}I_{m_\l}}\, ,$$
it follows that $f(\lt)A f(\overline{\lt})$ is represented by the matrix $(\lt^{m_i-m_j}A_i^j).$ Now write $\al(z)
\,=\,zA,$ where $A\,:=\,\al(1).$ Then the $f$-covariance condition yields
$$\lt A\,=\,\al(\lt)\,=\,f(\lt)A f(\overline{\lt})\, .$$
Thus
$$\lt A_i^j\,=\,\lt^{m_i-m_j}A_i^j$$
for all $\lt\,\in\,\Tl.$ This shows that the condition
$A_i^j\,\ne\, 0$ implies that $m_i-m_j\,=\,1.$

Similarly, consider $\bl(z)\,=\, \overline{z} B,$ where $B\,:=\,\bl(1).$ Then the $f$-covariance condition yields
$$\overline{\lt} B\,=\,\bl(\lt)\,=\,f(\lt)B f(\overline{\lt})\, .$$
Thus we have
$$\overline{\lt} B_i^j\,=\,\lt^{m_i-m_j}B_i^j$$
for all $\lt\,\in\,\Tl.$ This shows that the condition $B_i^j\,\ne\, 0$ implies that $m_i-m_j\,=\,1.$ 

It follows that for all `gaps' $m_i\,<\,m_{i+1}$ of size $>\,1$ we have $A_{i+1}^i\,=\,0\,=\,B_i^{i+1}.$ Thus
there is a decomposition
\begin{equation}\label{34}
\Cl^N\,=\,W^1\oplus\,\cdots \,\oplus W^k
\end{equation}
such that $A,\,B$ are block-diagonal matrices with respect to the decomposition in
\eqref{34} and each direct summand $W$ in \eqref{34} consists of a connected chain of weight spaces
$$W=\S_{i=0}^\l V_{m+i}$$
for a fixed integer $m.$

It would be sufficient to consider each summand $W$ separately. The restrictions of $A,\,B$ to $W$ give
rise to block-matrices 
$$\begin{pmatrix}E_0&B^1&0&0&0\\ A_1&E_1&B^2&0&0\\ \vdots&\ddots&\ddots&\ddots&\vdots\\ 0&0&A_{\l-1}&E_{\l-1}&
B^\l\\ 0&0&0&A_\l&E_\l\end{pmatrix}.$$

Thus the moduli space consists of maps $$V_{m+i-1}\, \stackrel{A_i}{\longrightarrow}\,
V_{m+i},\ \ V_{m+i}\, \stackrel{B_i}{\longrightarrow}\,V_{m+i-1},$$ for $1\,\le\, i\,\le\,\l,$ modulo
joint conjugation
$$\t A_i\,=\,g_i A_i g_{i-1}^{-1},\ \ \t B^i\,=\,g_{i-1}B^i g_i^{-1},\ \ 1\,\le\, i\,\le\, \l\, ,$$
where $(g_i)\in\prod_{i=0}^\l {\rm GL}(V_{m+i})$ is arbitrary. We note that this is
also the {\bf quiver variety} of the oriented graph
\begin{equation}\label{35}
\xymatrix{V_{m+0}\ar[r]^{A_1}&V_{m+1}\ar[r]^{A_2}&\cdots \ar[r]^{A_{\l-1}}&V_{m+\l-1}\ar[r]^{A_\l}&V_{m+\l}}
\end{equation}
whose double is
\begin{equation}\label{36}
\xymatrix{V_{m+0}\ar@/{}^{.5pc}/[r]^{A_1}&V_{m+1}\ar@/{}^{.5pc}/[r]^{A_2}\ar@/{}^{.5pc}/[l]^{B^1}&
\cdots\ar@/{}^{.5pc}/[r]^{A_{\l-1}}\ar@/{}^{.5pc}/[l]^{B^2}&
V_{m+\l-1}\ar@/{}^{.5pc}/[r]^{A_\l}\ar@/{}^{.5pc}/[l]^{B^{\l-1}}&V_{m+\l}\ar@/{}^{.5pc}/[l]^{B^\l}}.
\end{equation}
For any quiver $Q$ with index set $I$ the double quiver $\overline{Q}$ gives rise to a moment map
$$\lm(x)\,:=\,\big(\S_{a_+=i}x_a x_{\overline{a}}-\S_{a_-=i}x_{\overline{a}}x_a\big)_{i\in I}
\,\in\, \prod_{i\in I}\mathfrak{gl}(V_i)\, .$$
Here $a$ runs over all (oriented) edges in $\overline{Q},$ with head $a_+$ and tail $a_-,$ and $\overline{a}$ denotes the
opposite edge. In our case 
$I\,=\,\{0,\,1,\,\cdots,\,\l\}$, and for $x\,=\,(A_i,B^i)_{i=0}^\l$ we have
$$\lm(x)_0\,=\,B^1A_1-B^1A_1\,=\,0$$
$$\lm(x)_\l\,=\,A_\l B^\l-A_\l B^\l\,=\,0$$
$$\lm(x)_i\,=\,A_iB^i+B^{i+1}A_{i+1}-A_iB^i-B^{i+1}A_{i+1}\,=\,0$$
for $0\,<\,i\,<\,\l.$ Thus the condition $\lm(x)\,=\,0$ is trivially satisfied, and by \cite[Theorem 1.1]{CB},
the quotient 
$${\rm Rep}(\overline{Q})/\prod_{i=0}^\l {\rm GL}(V_{m+i})$$ is a normal variety. By \cite[Theorem 1]{BP} the
ring of polynomial invariants of the Mumford's geometric invariant
theoretic quotient $\Cl^N/\!\!/G$ is generated by traces of oriented cycles of the quiver
in \eqref{9} of length $\le\, N^2.$

In the {\bf higher dimensional case}, a first step is to find matrices $A_i\,=\,\al(e_i),\ B_i\,=\,\bl(e_i)$
as in Lemma \ref{j}. Consider first the unit ball. Here the non-zero tripotents are the unit vectors $e\in Z
\,:=\,\Cl^d,$ regarded as a Jordan triple of rank 1. For $K\,=\, {\rm U}(d)$ we put
\begin{equation}\label{C0}
K_e\,:=\, \{\l\,\in\, K\,\mid\ \l e\,=\, e\}\,\approx\, {\rm U}(d-1)\, .
\end{equation}

\begin{proposition} Let $A,\,B\,\in\,{\mathfrak h}$ satisfy
\begin{equation}\label{C1}
Ad_{f(\l)}^H A\,=\, A,\ \, Ad_{f(\l)}^H B\,=\,B
\end{equation}
for all $\l\in K_e,$ and
\begin{equation}\label{C2}
Ad_{f(\lz)}^H A\,=\,\lz A,\ \ Ad_{f(\lz)}^H B\,=\,\overline{\lz} B
\end{equation}
for all $\lz\in\Tl$ (center of $K$). Then there exist unique $f$-covariant linear maps $\al\,:\,Z
\,\longrightarrow\,{\mathfrak h},\ \bl\,:\,\overline{Z}\,\longrightarrow\,{\mathfrak h}$ satisfying 
$$\al(e)\,=\,A,\ \ \bl(e)\,=\,B\, .$$
\end{proposition}

\begin{proof}
The (Shilov) boundary $S\,=\,\Sl^{2d-1}$ of the unit ball $D$ is the homogeneous space $S\,=\,K/K_e.$ Define a 
real-analytic mapping $$a\,:\,S\,\longrightarrow\,{\mathfrak h}$$ by $a(ke)\,:=\,Ad_{f(k)}^H A$ for all $k\,\in\, K.$ This is 
well-defined since for $k,\,k'\,\in\, K$ satisfying the condition $ke\,=\,k'e$; we have $\l
\,:=\,k^{-1}k'\in K_e$ and hence 
$Ad_{f(k)}^H A\,=\,Ad_{f(k')}^H A$ as a consequence of \eqref{C1}. By construction, we have $a(e)\,=\,A$ and 
$a(ks)\,=\,Ad_{f(k)}^H a(s)$ for all $k\,\in\, K,\ s\,\in\, S.$

Now consider the harmonic extension $\h a\,:\,D 
\,\longrightarrow\,{\mathfrak h}$ of $a,$ given by the Poisson integral. Then $\h a|_S\,=\,a$ and $\h a$ is also 
$f$-covariant, since the Poisson integral operator commutes with the action of $K.$ Moreover, the
condition in \eqref{C2} implies that
$$\h a(\lz z)\,=\,\lz\ \h a(z)$$
for all $\lz\,\in\,\Tl$ and $z\,\in\, D.$ Therefore $\h a$ is (the restriction of) a $\Cl$-linear map, again denoted by 
$a,$ on $Z.$ The anti-linear case is treated similarly.
\end{proof}

In the higher rank case, things are more complicated. Let $Z$ be an irreducible Jordan triple of rank $r,$ and let
$e_1,\,\cdots,\,e_r$ be a frame of minimal orthogonal tripotents. Given matrices $A_1,\,\cdots,\,A_r\,\in\,{\mathfrak h}$ we put
$$e_I\,:=\, \S_{i\in I}e_i,\ \ A_I\,:=\,\S_{i\in I}A_i\,\in\,{\mathfrak h}$$
for all subsets $I\,\subset\, \{1,\,\cdots,\,r\}.$ We seek conditions on $A_i$ such that there exists a (unique) $f$-covariant $\Cl$-linear map 
$$\al\,:\,Z\,\longrightarrow\,{\mathfrak h}$$ satisfying the condition
$\al(e_i)\,=\,A_i$ for all $1\,\le\, i\,\le\, r.$ If $I,\, J$ are two subsets of equal cardinality $|I|\,=\,|J|,$
then $e_I$ and 
$e_J$ are tripotents of equal rank, and there exists $k\,\in\, K$ such that $ke_I\,=\,e_J.$
A necessary condition for the $A_i$ is that
\begin{equation}\label{C3}Ad_{f(k)}^H A_I\,=\,A_J
\end{equation}
whenever $|I|\,=\,|J|$ and $k\,\in\, K$ satisfies
the condition $k e_I\,=\,e_J.$ This applies in particular for $I\,=\,J.$ Applying \eqref{C3} to $I\,=\,
\{1,\,\cdots,\,r\}$ we have
$$Ad_{f(k)}^H(A_1+\ldots+A_r)\,=\,A_1+\ldots+A_r$$
for all $k\,\in\, K$ fixing the maximal tripotent $e\,:=\,
e_1+\ldots +e_r.$ Now $S\,=\,K/K_e$ becomes the Shilov boundary of the (spectral) unit ball 
$D\,\subset\, Z$ (a proper subset of the full boundary if $r\,>\,1$) and, as in the rank 1 case, we may define
a $f$-covariant real-analytic mapping $a\,:\,S\,\longrightarrow\,{\mathfrak h}$ by putting
$$a(ke)\,:=\,Ad_{f(k)}^H(A_1+\cdots+A_r)$$
for all $k\,\in \,K.$ Consider its harmonic extension $\h a\,:\,D\,\longrightarrow\,{\mathfrak h}$ given by the higher-rank
analogue of the Poisson integral \cite{Ko}. A Cartan subspace of ${\mathfrak k}$ is given by the commuting linear
vector fields
$$\sqrt{-1}\ \{e_i;\,e_i;\, z\}\f{\dl}{\dl z}$$
where $1\,\le\, i\,\le\, r.$ Integrating these vector fields we
obtain a holomorphic torus action $$(\lt,\,z)\,\longmapsto\,\lt\cdot z$$ of $\lt\,\in\,\Tl^r$ on $D$ such that
$$(\lt_1,\,\cdots,\,\lt_r)e_i\,=\,\lt_i\ e_i$$
for all $i.$ A second necessary condition, generalizing \eqref{C2}, is the following:
$$Ad_{f(\lt)}^H A_i\,=\,\lt_i A_i$$
for all $\lt\in\Tl^d\,\subset\, K.$ Imposition this condition leads to a unique $\Cl$-linear extension $$a\,:\,Z
\,\longrightarrow\,{\mathfrak h}$$ which is still $f$-covariant. The anti-linear case of matrices $B_i$ is treated similarly. 

In this way, the moduli space of pure invariant connections gives rise to varieties of commuting matrix tuples (modulo joint conjugation), on which there is a rich literature \cite{GS}. A detailed study is reserved for a future publication.

\section*{Acknowledgements}

The first author thanks Philipps-Universit\"at Marburg for hospitality while a part of
the work was carried out. He is partially supported by a J. C. Bose Fellowship.

\end{document}